\documentclass[a4paper,12pt,reqno]{article}

 \usepackage[english]{babel}
  \usepackage{amsmath,amsfonts,amssymb,amsthm,bbm}
   \usepackage{graphics,epsfig,psfrag} 
    \usepackage{esint}
   \usepackage{paralist,subfigure}
   \usepackage[dvipsnames]{xcolor}
   \usepackage{hyperref} 

    \usepackage[active]{srcltx} 
    \usepackage{color,soul} 
    \usepackage[latin1]{inputenc}
    \usepackage[OT1]{fontenc}
    \usepackage{authblk}

\newtheorem{theorem}{Theorem}[section]
\newtheorem{cor}[theorem]{Corollary}
\newtheorem{lemma}[theorem]{Lemma}
\newtheorem{prop}[theorem]{Proposition}

\theoremstyle{definition}
\newtheorem{definition}[theorem]{Definition}
\newtheorem{req}{Remark}

\newtheorem{question}{Question}

\def\N{\mathbb{N}}

\def\R{\mathbb{R}}

\let\O=\Omega
\let\e=\varepsilon

\let\t=\tilde
\let\ol=\overline
\let\ul=\underline
\let\.=\cdot
\let\0=\emptyset

\let\mc=\mathcal

\def\1{\mathbbm{1}}

\def\l{\lambda_1}

\newenvironment{formula}[1]{\begin{equation}\label{#1}}
                       {\end{equation}\noindent}

\def\Fi#1{\begin{formula}{#1}}
\def\Ff{\end{formula}\noindent}

\title{\bf  Qualitative properties of spatial epidemiological models.}

\author[]{Romain {\sc Ducasse}}
\affil[]{Institut Camille Jordan, Universit\'e Claude Bernard Lyon 1}				

\begin{document}

\date{}
\maketitle



\begin{abstract}
We study the qualitative properties of a spatial diffusive heterogeneous SIR model, that appears in mathematical epidemiology to describe the spread of an infectious disease in a population. The model we consider consists in a system of parabolic PDEs.

In the first part of the paper, we give a criterion that ensures whether or not an epidemic propagates in a given population. We show how the features of the disease and of the population (rates of infection and of recovery, localisation and diffusivity of individuals) influence the propagation of the epidemic. In particular, we prove that there are situations where ``slowing down" the individuals can trigger an epidemic that would not propagate otherwise.

In the second part of the paper, we show how the spatial diffusive SIR model qualitatively differs from the usual, purely temporal, SIR model.
\end{abstract}

\section{Introduction}\label{sec intro}

\subsection{SIR models}

An epidemic is an outbreak of a disease that affects a great number of individuals in a short range of time. Many mathematical models were introduced to study the spread and the outcome of epidemics and to help inform public health interventions. \\

The study of deterministic mathematical models in epidemiology dates back to the seminal works of W. O. Kermack and A.~G.~McKendrick \cite{KmcK1, KmcK2, KmcK3} in $1927$. They introduced several models designed to describe the temporal development of a disease; the most famous being probably the \emph{SIR} (Susceptible - Infectious - Recovered) model. It is a \emph{compartmental model}, that is, the population under consideration is divided into compartments: the \emph{Susceptibles} are the individuals that are untouched by the disease, they can be contaminated by the \emph{Infectious}. The infectious become \emph{Recovered} after some time. The recovered individuals can not become susceptible again, either because they have permanent immunity, or because they are dead.

The SIR model of Kermack and McKendrick consists in the following system of ODEs:

\begin{equation}\label{SIR}
\left\{
    \begin{array}{rlll}
    \dot S(t) &= - \alpha SI,& \quad    &t>0,  \\
   \dot I(t) &= \alpha SI - \mu I,& \quad    &t>0.
   \end{array}
\right.
\end{equation}
The functions $S(t),I(t)$ represent the fraction of the population that is susceptible and infectious respectively, at time $t>0$. The susceptible become infectious following a law of mass-action, with a rate $\alpha  I$, where $\alpha>0$ is a constant parameter that measures the transmission efficiency of the disease. The more infectious individuals there are in the population, the more likely it is for a susceptible to get contaminated, and then to become infectious.

The infectious individuals have a recovery rate $\mu >0$ (then $\frac{1}{\mu}$ is the life expectancy of an infectious).\footnote{The system \eqref{SIR} is sometimes completed by a third equation $\dot R(t) = \mu I(t)$, where $R(t)$ is the density of recovered individuals. They do not play any role in the dynamic of the system.
 }\\

The SIR model \eqref{SIR} received a lot of attention, due to the fact that it reflects several qualitative properties of epidemics. One of its most important properties is the \emph{threshold phenomenon}, that states that, if the quantity
$$
R_0 := \frac{\alpha S_0}{\mu}
$$
is strictly larger than $1$, then an epidemic propagates, in the sense that the introduction of even a single infectious individual in the population (sometimes called the \emph{patient zero}, or the \emph{index case}) triggers the contamination of ``many other" individuals. If $R_0 \leq 1$, the epidemic fades off. The notions of propagation and fading-off are precisely defined below.

The number $R_0$, called the \emph{basic reproduction number}, represents the ration between the number of newly infectious individuals and the number of newly recovered individuals at initial time. It relates the features of the disease (here, the mortality and the contamination rates) and of the population (the number of susceptible individuals).

The SIR model was also useful in understanding the concept of \emph{herd immunity}, and how to reach it through mass vaccination. We refer to \cite{Mu1, Mu2} and the references therein for more details.\\

An important question, when modelling a phenomenon, is to find the adequate balance between the level of precision with which we describe it and the mathematical complexity of the resulting model. From that perspective, the SIR model \eqref{SIR} is rather simple, it consists in two coupled first-order ODEs. This is a consequence of the fact that it describes the evolution of the disease at a coarse scale: it does not take into the dynamics of the individuals at the microscopic scale. The model \eqref{SIR} is \emph{macroscale}.

Because of that, it is often thought that \eqref{SIR} is accurate only to describe the evolution of a disease in a population located in one ``small" site, where contacts between individuals are extremely frequent, such as a very dense city for instance.\\

If one wants to describe the propagation of a disease into a region, a country, or the whole world, the setting is different, and it is then natural to take into account in the model spatial, microscopic, effects: movement of individuals (migration, diffusion), spatial distribution of the population...\\


Starting from the SIR model \eqref{SIR}, we can build a simple \emph{microscale} model for the spatial spread of an epidemic by adding diffusion terms. These terms will reflect the spatial dispersal of individuals at the microscopic scale. Doing so, several authors (we recall some past results in Section \ref{past results}) were lead to study the following system of PDEs:

\begin{equation}\label{systH}
\left\{
    \begin{array}{rlll}
    \partial_t S(t,x) &= d_S \Delta S(t,x) - \alpha SI,& \quad    &t>0, \ x\in \O,  \\
   \partial_t I(t,x) &= d_I \Delta I(t,x) + \alpha SI - \mu I,& \quad    &t>0, \ x\in \O, \\
   \partial_\nu S(t,x) &= \partial_\nu I(t,x) =0,& \quad &t>0, \ x\in \partial\O.
    \end{array}
\right.
\end{equation}
In this system, the domain $\O \subset \R^N$ is an open connected bounded set of class $C^2$, and $\nu$ is the unit outward normal vector field to $\O$.

The functions $S(t,x),I(t,x)$ represent the densities of susceptible and infectious individuals respectively, at time $t>0$ and at position $x\in\O$. The individuals move randomly following a Brownian motion on the domain, which is reflected in the equations by the presence of Laplace operators. The quantities $d_S,d_I >0$ are the diffusivities of the susceptible and infectious individuals respectively. They represent the amplitude of the Brownian motions of the individuals.

The Neuman boundary condition accounts for the fact that the individuals that reach the boundary bounce back into the domain following a Descartes reflection law. Finally, one has to to complete this system with an initial condition $(S_0(\cdot),I_0(\cdot))$ representing the initial spatial distribution of susceptible and infectious individuals. Observe that, if $S_0,I_0$ are constant over $\O$, then \eqref{systH} boils down to \eqref{SIR}.\\

The model \eqref{systH} takes into account spatial effects, unlike \eqref{SIR}, but it makes the assumption that the underlying phenomena are \emph{homogeneous}. It is sometimes natural to study heterogeneous versions of \eqref{systH}. Indeed, the infection and recovery rates can vary from places to places (for instance if there are isolated or quarantine zones), the diffusion of individuals can change according to the location (reflecting the geography of the territory for instance).

A way to take into account heterogeneous effects in the system is to modify~\eqref{systH} into the following:

 \begin{equation}\label{syst}
\left\{
    \begin{array}{rlll}
    \partial_t S(t,x) &= \nabla \cdot ( A_S(x) \nabla S)(t,x) - \alpha(x)SI,& \quad    &t>0, \ x\in \O,  \\
   \partial_t I(t,x) &= \nabla \cdot ( A_I(x) \nabla I)(t,x) + \alpha(x)SI - \mu(x)I,& \quad    &t>0, \ x\in \O, \\
   \nu \cdot A_S\nabla S(t,x) &= \nu\cdot A_I \nabla I(t,x) =0,& \quad &t>0, \ x\in \partial\O.
    \end{array}
\right.
\end{equation}
In this model, the movement of the individuals are given by diffusion matrices $A_S,A_I$. In the whole paper, these matrices are assumed to be elliptic and of class $C^1(\ol\O)$. The per capita rate of infection $\alpha$ and the rate of recovery $\mu$ are assumed to be continuous and strictly positive functions of $x$ on $\ol \O$.

In the sequel, we will refer to \eqref{syst} as \emph{the diffusive model}. \\







The first question we adress in the paper is the following:

\begin{question}\label{q1}
Under which conditions on the features of the epidemic (rates of contamination and of recovery), and of the population (diffusivity and initial localisation) does the apparition of a disease in a population triggers the spread of an epidemic?
\end{question}

When dealing with the model \eqref{SIR}, we recalled above that the necessary and sufficient condition for the epidemic to propagate (in a sense defined below) is to have $R_0>1$. We generalize this result to the diffusive model \eqref{syst} by proving that it also exhibits a threshold phenomenon: there is a quantity (\emph{a priori} different from $R_0$) whose value indicates whether or not the epidemic propagates or fades off.

We will study how this quantity depends on the parameters of the system $\alpha, \mu, A_I, A_S, S_0$, this will allow us to exhibit some qualitative properties of the SIR system \eqref{syst}. In particular, we will show that decreasing the diffusion of the infectious individuals can trigger the spread of an epidemic, and increasing this diffusion can block the epidemic. \\

We explained above that the model \eqref{syst} was built by adding diffusion terms in the macroscale model \eqref{SIR}. This strategy is very common in population dynamics, in chemistry, in mathematical neuroscience, and, of course, in epidemiology. This strategy generally increases the mathematical complexity of the model: it turns ODE systems into PDE systems.

However, a naive - yet also natural - way to account for the microscopic features of the epidemic starting from \eqref{SIR} is simply to \emph{average every spatial quantities}. More precisely, if we want to describe the evolution of a disease in a region $\O$, where the susceptible and infectious individuals are initially distributed according to the densities $S_0(\cdot), I_0(\cdot)$ and where the contamination and recovery rates are the functions $\alpha(\cdot),\mu(\cdot)$, then it is tempting to define the averaged quantities $\ol S_0 = \fint_\O S_0, \ol I_0 = \fint_\O I_0, \ol \alpha = \fint_\O \alpha, \ol\mu =\fint_\O \mu$,\footnote{The notation $\fint_\O$ denotes the spatial average, that is, for $u \in L^1(\O)$, $\fint_\O u := \frac{1}{\vert \O\vert}\int_\O u$, with $\vert \O\vert$ the measure of $\O$.} and to expect the total number of susceptible and infectious individuals at time $t>0$ to be solutions of
\begin{equation}\label{SIRav}
\left\{
    \begin{array}{rlll}
    \dot {\mc S}(t) &= - \ol\alpha \mc S \mc I,& \quad    &t>0,  \\
   \dot {\mc I}(t) &= \ol\alpha \mc S\mc I - \ol\mu \mc I,& \quad    &t>0,\\
   \mc S(0) &= \ol S_0, \quad \mc I(0) = \ol I_0,& &
   \end{array}
\right.
\end{equation}
The system \eqref{SIRav} is simply the macroscale system \eqref{SIR} but where the microscopic effects are averaged to be turned into macroscopic quantities. If $\alpha,\mu,S_0$ and $I_0$ are constant, then the averaged model \eqref{SIRav} is equivalent to the model \eqref{syst} in the sense that $(\mc S,\mc I) = (\fint_\O S,\fint_\O I)$, where $(\mc S,\mc I)$ and $(S,I)$ are the solutions of \eqref{SIRav} and \eqref{syst} respectively (the latter is independent of the $x$ variable). As soon as any of the quantity $\alpha, \mu, S_0, I_0$ is not constant, this is not true anymore.

In the sequel, we shall refer to the model \eqref{SIRav} as \emph{the averaged model}. The second question we study in the paper is then the following:

\begin{question}\label{q2}
Can we compare the predictions of the diffusive model \eqref{syst} with those of the averaged model \eqref{SIRav}?
\end{question}

We will investigate, for $A_S,A_I,\alpha,\mu,S_0,I_0$ given, how \eqref{SIRav} and \eqref{syst} differ in two aspects. First, we will show that the averaged model always ``underestimate" the risk that the epidemic propagates, in the sense that, if the epidemic propagates in \eqref{SIRav}, it also propagates for \eqref{syst}, but we will exhibit situations where the reciprocal is false. 

Then, we will compare the predictions of the models in what concerns the number of individuals left untouched after the epidemic, or, what is equivalent, the number of contaminated individuals. We will in particular show that, when $\alpha,\mu,S_0$ are constant the averaged model \eqref{SIRav} always underestimate the number of casualties compared to the homogeneous diffusive model \eqref{systH}. This difference, however, disappears when $I_0$ is ``small".\\


 


We conclude this section by defining what it means for an epidemic to propagate in the SIR models above.

\begin{definition}\label{def}
Let $\alpha,\mu, S_0$ be positive and continuous on $\ol \O$.
\begin{itemize}
    \item We say that the \emph{epidemic propagates} for the SIR model \eqref{syst} if and only if there is $\e >0$ such that, for every $I_0 \in C^0(\ol\O)$, $I_0\geq 0$, $I_0 \not\equiv 0$, the solution $(S(t,x),I(t,x))$ of \eqref{syst} arising from the initial datum $(S_0,I_0)$ satisfies
	$$
	\int_\O S_\infty \leq \int_\O S_0-\e,
	$$
	where $S_\infty := \lim_{t\to+\infty} S(t,x)$.

	\item We say that the \emph{epidemic fades off} for the SIR model \eqref{syst} if and only if, for every $\e>0$, there is $\delta>0$ such that, for very $I_0$ such that $I_0\geq 0$, $I_0\not\equiv 0$ and $\vert\vert I_0\vert\vert_{L^\infty(\O)} \leq \delta$, the solution $(S(t,x),I(t,x))$ of \eqref{syst} arising from the initial datum $(S_0,I_0)$ satisfies
	$$
	\int_\O S_{\infty} \geq \int_\O S_0 - \e,
	$$
	where $S_{\infty} := \lim_{t\to+\infty}S(t,x)$.
\end{itemize}
\end{definition}
Let us say a word on this definition. The quantity $\int_\O S_0$ represents the total number of susceptible individuals at initial time, and $\int_\O S_\infty$ represents the final number of susceptible individuals. The above definition says that the epidemic propagates if the number of susceptible individuals is strictly decreased after the epidemic has passed, even when there are infinitely few infectious individuals at the initial time. Considering situations with a very small amount of infectious individuals at the initial time is meaningful from the modeling point of view: when a new disease appears, there are usually very few infectious individuals, and they are usually localized, and the question is whether or not a single individual, the patient zero, can be the source of an epidemic.

The notions of propagation and of fading-off for the epidemic given in Definition~\ref{def} also holds true for the homogeneous system \eqref{systH} and for the ODE systems~\eqref{SIR} and \eqref{SIRav}; in this last case, we say that the epidemic propagates if there is $\e>0$ such that, for every $I_0>0$, the solution $(S(t),I(t))$ of \eqref{SIR} or \eqref{SIRav} satisfies $S_\infty \leq S_0 -\e$, where $S_\infty = \lim_{t\to +\infty}S(t)$. The definition of the fading-off is similar.


\subsection{Review of some results on the SIR models}\label{past results}

We present in this section some results concerning the SIR systems \eqref{SIR}, \eqref{systH} and~\eqref{syst}. A more comprehensive presentation, together with modeling and biological discussions, can be found in \cite{Mu1, Mu2}.\\

The main result concerning the ODE system \eqref{SIR} is the following:
\begin{theorem}[\cite{KmcK1}]\label{th SIR}
Let $\alpha,\mu,S_0$ be positive constants. Then
\begin{itemize}
    \item If $\frac{\alpha S_0}{\mu}>1$, the epidemic propagates in the SIR model \eqref{SIR}, in the sense of Definition \ref{def}.
    
    \item If $\frac{\alpha S_0}{\mu}\leq 1$, the epidemic fades off in the SIR model \eqref{SIR}, in the sense of Definition \ref{def}.
\end{itemize}
    In addition, if $(S(t),I(t))$ is the solution of \eqref{SIR} arising from the initial datum $(S_0,I_0)$ with $I_0\geq 0$, then,
    $$
    S(t) \underset{t \to +\infty}{\longrightarrow} S_\infty, \quad I(t) \underset{t \to +\infty}{\longrightarrow} 0,
    $$
    where $S_\infty$ is the unique real number such that $S_\infty \leq S_0$ and
    \begin{equation}\label{eq th}
    \frac{\alpha}{\mu}S_\infty - \ln(S_\infty) = \frac{\alpha}{\mu}S_0 - \ln(S_0) + \frac{\alpha}{\mu}I_0.     
    \end{equation}
\end{theorem}
This result summarizes the results of Kermack and McKendrick on the model~\eqref{SIR}. First, it establishes the threshold effect mentioned above: the epidemic propagates if and only if the quantity $\frac{\alpha S_0}{\mu}$ is greater than $1$.

The second part of the result indicates how to compute $S_\infty$, the number of individuals left untouched after the epidemic.\\

Let us give a quick idea about how to prove Theorem \ref{th SIR}. The key point is that the quantity
$$
\mc E(t) := \frac{\alpha}{\mu}S(t) - \ln(S(t)) + \frac{\alpha}{\mu} I(t)
$$
is conserved along the evolution of the system. In addition, one can prove that $I(t)$ goes to zero and that $S(t)$ converges to a positive constant $S_\infty$ as $t$ goes to $+\infty$. Therefore, using $\lim_{t\to+\infty}\mc E(t) = \mc E(0)$, one gets \eqref{eq th}.

Observing that the function $x \mapsto \frac{\alpha}{\mu}x - \ln(x) $ is strictly convex, reaches its unique minimum at $x = \frac{\mu}{\alpha}$, and goes to $+\infty$ as $x$ goes to zero and to $+\infty$, we see that the equation \eqref{eq th} admits two solutions for $S_\infty$, but because $S(t)$ is non-increasing (because $\dot S(t) \leq 0$), we have $S_\infty \leq S_0$ and \eqref{eq th} uniquely defines $S_\infty$.

The rest of the theorem then comes by observing that, for $S_0$ fixed, the quantity $S_\infty$ given by \eqref{eq th} converges, as $I_0>0$ goes to zero, to $S_0$ if $S_0\leq \frac{\mu}{\alpha}$, while it converges to the unique solution of $f(S_\infty) = f(S_0)$ such that $S_\infty < S_0$ if $S_0> \frac{\mu}{\alpha}$.

We refer to \cite{formulation, KmcK1} for the details of the proof.\\

When considering spatial systems like \eqref{systH}, \eqref{syst}, the situation is more involved. Indeed, what makes the analysis possible for \eqref{SIR} is that we know a quantity that is conserved along the evolution; this is not the case anymore for spatial models.\\

The spatial homogeneous system \eqref{systH} is studied by Hosono and Ilyas in \cite{HI}, with $\O = \mathbb{R}$. They show the existence of \emph{traveling waves}, that is, solutions of the form $(S(t,x),I(t,x)) = (p(x - ct),q(x-ct))$, with $q(-\infty)=q(+\infty)=0$ and $p(+\infty) = S_0$ and $p(-\infty)= S_\infty$. Here, $c\in \R$ is the speed of the wave.

The existence of waves is particularly interesting as it allows to define a notion of \emph{speed for the epidemic}.

The main result of \cite{HI} is that, if $\frac{\alpha S_0}{\mu} >1$, there exists $S_\infty>0$ such that there are traveling waves with speed $c$, for every $c\geq c^\star := 2\sqrt{d_I(\alpha S_0 - \mu)}$. There are no traveling waves when $c<c^\star$ or when $\frac{\alpha S_0}{\mu} \leq 1$.

The authors of \cite{HI} use crucially the fact that the system is homogeneous to rewrite the equations satisfied by the profiles of the waves (the functions $p$ and $q$) and to use a phase-plane analysis. 

Let us emphasize that the quantity $S_\infty$ is not explicit in the paper \cite{HI}.

The results of \cite{HI}, the existence and the computation of the speed of traveling waves, are very similar to a celebrated result of Kolmogorov, Petrovski and Piskunov concerning the existence of traveling waves for reaction-diffusion equations, see \cite{KPP}.\\

The heterogeneous system \eqref{syst} does not allow for a phase pane analysis, and its analysis is for now much less advanced, except in the specific case where the diffusion of the susceptible individuals is zero, that is, when $A_S \equiv 0$. 
In this case, one can do some change of variable : we can show that the function $u(t,x) := -\ln\left(\frac{S(t,x)}{S_0(x)}\right)$ solves an integral equation of the form
\begin{equation}\label{eq integ}
u(t,x) = \int_0^t\int_\O H(t,x,y)g(u(t-\tau,y))dyd\tau + f(t,x), \quad t>0, \ x\in \O.
\end{equation}
The function $u$ is sometimes called the \emph{strenght of infection}. The functions $H,g,f$ encode the features of the epidemic and of the population, and can be computed from $A_I,\alpha,\mu,S_0,I_0$.

This approach was used originally in the homogeneous framework by O. Diekmann and H. Thieme independently, see \cite{Di1, T1}. They prove that there is a threshold phenomenon for the equation \eqref{eq integ}, and also that there exist traveling wave solutions to~\eqref{eq integ}. 

The threshold phenomenon for integral equations of the form \eqref{eq integ} was extended to more general heterogeneous frameworks by H. Inaba \cite{I} when $\O$ is bounded.

Using a related change of variable, A. Ducrot and T. Giletti prove the existence of traveling waves for \eqref{syst} with periodic heterogeneities in \cite{DG}. The author of the present paper considered some models with non-local interactions and periodic heterogeneities in \cite{D}.\\

Finally, we want to conclude this section with a word about other similar epidemiological models. As we already mentioned, the SIR model is a compartmental model, with three compartments: susceptibles, infectious, recovered.

There exist many other models. For instance, the SIS model consists in considering two compartments only: the susceptibles and the infectious. The susceptibles become infectious as before, but the infectious do not recover but become susceptible again. This reflects a waning of immunity. The simplest SIS system reads:

\begin{equation*}
\left\{
    \begin{array}{rlll}
    \dot S(t) &= - \alpha SI + \mu I,& \quad    &t>0,  \\
   \dot I(t) &= \alpha SI - \mu I,& \quad    &t>0.
   \end{array}
\right.
\end{equation*}
Allen, Bolker, Lou and Nevai consider in \cite{Allen} a spatial SIS model of the form
\begin{equation*}\label{SIS}
\left\{
    \begin{array}{rlll}
    \partial_t S(t,x) &= d_S \Delta S(t,x) - \alpha \frac{SI}{S+I} + \mu I,& \quad    &t>0, \ x\in \O,  \\
   \partial_t I(t,x) &= d_I \Delta I(t,x) + \alpha \frac{SI}{S+I} - \mu I,& \quad    &t>0, \ x\in \O, \\
   \partial_\nu S(t,x) &= \partial_\nu I(t,x) =0,& \quad &t>0, \ x\in \partial\O.
    \end{array}
\right.
\end{equation*}
They study the existence of stationary states and their stability. \\

Let us mention that, although the SIS models and the SIR models look very similar, a crucial difference between them is that it is easier to find the stationary solutions for the former. Indeed, in the SIS models, there is a mass conservation property that fails to hold true for SIR systems, this makes the analysis more intricate.

\subsection{Results of the paper}

This paper is dedicated to the study of the long-time behavior of spatial SIR models. In a first part, we answer Question \ref{q1} by proving that the model \eqref{syst} exhibits a threshold phenomenon. More precisely, we show that there is a quantity - given by the principal eigenvalue of an elliptic operator - whose value determines whether or not the epidemic propagates. We study how this eigenvalue depends on the parameters of the system, so that we can understand how the features of the model influence the way the epidemic propagates.

In a second part, we investigate Question \ref{q2}. We compare the diffusive model \eqref{syst} with the averaged model \eqref{SIRav}. We show how their predictions concerning the propagation of the epidemic and the impact of the epidemic on the population differ. \\

We assume in the whole paper without further notice that $\O \subset \R^N$ is an open connected bounded set of class $C^2$, that the functions $\alpha,\mu$ in \eqref{syst} are strictly positive and continuous on $\ol\O$, that the matrix fields $A_S,A_I$ are of class $C^1$ and strictly elliptic on $\ol\O$. The initial data $(S_0,I_0)$ will always be such that $S_0, I_0\geq 0$ and $S_0, I_0$ continuous on $\ol \O$. Under these hypotheses, there exist a unique solution $(S(t,x),I(t,x))$ of \eqref{syst}, with $S,I$ of class $C^1$ for $t>0$ and $C^2$ for $x \in \O$. We refer to \cite{DH} for a proof of this fact.\\

Our first result concerns the threshold phenomenon, that is, it gives a criterion that says whether or not the epidemic propagates or fades off (in the sense of Definition \ref{def}).

Let us recall the definition of the principal eigenvalue of an elliptic operator. For $A$ of class $C^1$ and strictly elliptic and $V$ continuous on $\ol\O$, let $L$ be the elliptic operator such that, for $\phi \in C^2(\O)$,
$$
L \ : \ \phi \mapsto -\nabla( A\nabla \phi) - V\phi.
$$
The principal eigenvalue of $L$ with Neuman conormal boundary conditions on $\partial \O$ is the unique $\lambda \in \R$ such that there is $\phi >0$ continuous on $\ol \O$, $\nu \cdot A\nabla \phi =0$ on $\partial \O$ (where $\nu$ is the unit normal outward vector field on $\O$), and such that $L\phi = \lambda \phi$. The existence of a principal eigenvalue for elliptic operators comes from the Krein-Rutman theorem, see \cite{KR}. In the sequel, all elliptic eigenproblems will be understood with conormal boundary conditions.

\begin{theorem}\label{th threshold}
Let $A_I$ be elliptic and of class $C^1(\ol\O)$, let $\alpha,\mu,S_0$ be continuous and positive on $\ol\O$. 
Let $\l$ be the principal eigenvalue of the elliptic operator
$$
\phi \mapsto -\nabla (A_I \nabla \phi) -\left( \alpha \fint_\O S_0 - \mu \right)\phi.
$$
Then
\begin{itemize}
    \item If $\l <0$, the epidemic propagates in \eqref{syst}.
    \item If $\l>0$, the epidemic fades off in \eqref{syst}.
\end{itemize}
\end{theorem}
Let us explain heuristically this result. The question to find whether or not the epidemic propagates requires to find the long-time behavior of the solutions of \eqref{syst}. When dealing with monotone systems, a standard approach is to study the stability of stationary solutions. The problem here is that \eqref{syst} has many stationary solutions: every couple $(\sigma,0)$, with $\sigma \in \R$ is a stationary solution, and many of these are linearly stable. In addition, \eqref{syst} is not monotonous.

However, we are interested in the situation where $I_0$ is small. In this case, we can expect $I$ to be small for a ``long time". Then, $\partial_t S \approx \nabla (A_S \nabla S)$, i.e., the evolution of $S$ is mostly governed by the diffusion process. Therefore, by the time $I$ gets large enough, we would have $S\approx \fint S_0$. Hence, $\partial_t I \approx \nabla (A_I \nabla I) +(\alpha \fint S_0 -\mu)I$, and then we can expect the dynamic of $I$ to be given by the principal eigenvalue of the elliptic operator $-\nabla A_I \nabla  - (\alpha \fint S_0 -\mu)$.\\

Theorem \ref{th threshold} does not say anything about the case where $\l=0$. We leave this as an open question; however, analogy with the case studied in \cite{D} (where $A_S \equiv 0$) suggests that we should have fading-off in this case.\\


Theorem \ref{th threshold} allows to find several qualitative properties of the system \eqref{syst}. Indeed, because the sign of $\l$ determines whether or not the epidemic propagates, understanding how $\l$ depends on the parameters of the system will indicate how the features of the model influence the propagation of the epidemic. 

This is the object of the next proposition, where we make explicit the dependence of the principal eigenvalue $\l$ by denoting it as a function of the parameters of the system.
\begin{prop}\label{prop quali}

For $A_I$ elliptic of class $C^1$ and for $\alpha,\mu,S_0$ in $C^0(\ol\O)$, let $\lambda_I(A_I,\alpha,\mu,S_0)$ denote the principal eigenvalue of the operator 
$$
\phi \mapsto -\nabla\cdot(A_I \nabla \phi) - \left(\alpha \fint_\O S_0 -\mu\right)\phi.
$$
Then,
\begin{itemize}
    \item $\lambda_1$ is nondecreasing with respect to $A_I$, strictly increasing with respect to $\mu$ and strictly decreasing with respect to $\alpha,\fint S_0$, i.e., if $\t A_I \geq A_I$, $\t \mu \geq \mu$, $\t \alpha \leq \alpha$, $\fint \t S_0 \leq \fint_\O S_0$, then
    $$
    \lambda_I(A_I,\alpha,\mu,S_0) \leq \lambda_I(\t A_I,\t \alpha,\t \mu,\t S_0),
    $$
    and this inequality is strict as soon as either $\t \mu > \mu$, $\t \alpha < \alpha$ or $\fint \t S_0 < \fint_\O S_0$. In addition, if $\alpha\fint_\O S_0 -\mu$ is not constant, the monotonicity with respect to $A_I$ is also strict.

    \item $\lambda_1$ converges when the diffusion of the infectious goes to infinity, that is, denoting $I_d$ the identity matrix:
    $$\lambda_I(d_I I_d,\alpha,\mu,S_0) \underset{d_I \to +\infty}{\nearrow}\fint \mu - \fint \alpha \fint S_0.$$

     \item $\lambda_1$ converges when the diffusion of the infectious goes to zero:
    $$\lambda_I(d_II_d,\alpha,\mu,S_0) \underset{d_I \to 0}{\searrow} \min_{x\in\O}\left\{\mu(x)-\alpha(x)\fint S_0\right\}.$$
\end{itemize}
\end{prop}
Combining Proposition \ref{prop quali} with Theorem \ref{th threshold} allows to derive some qualitative properties for the spread of epidemics. In particular, we directly see that, because $\lambda_1$ does not depend on the diffusivity of the susceptible individuals, then only the diffusivity of the infectious individuals plays a role on the propagation of the population. This may be somewhat surprising.
However, this is somewhat reminiscent of the result of  Hosono and Ilyas from \cite{HI} recalled in Section \ref{past results}: the critical speed of traveling waves for the homogeneous SIR system \eqref{systH} with $\O=\mathbb{R}$ is independant of $d_S$.

Proposition \ref{prop quali} also suggests that the larger the recovery rate or the population of susceptibles and the smaller the contamination rate, the less likely the epidemic propagates, which is not surprising.

Combining Proposition \ref{prop quali} with Theorem \ref{th threshold}, we also get the following more interesting result:
\begin{cor}\label{cor prop}
Let $\alpha,\mu,S_0$ be positive and continuous on $\ol \O$. If $\alpha,\mu$ are such that

    $$\frac{\fint_\O\alpha\fint_\O S_0}{\fint_\O \mu}<1<\max_{x\in\O}\left\{\frac{\alpha(x)\fint_\O S_0}{\mu(x)}\right\},$$
 then there is $d^\star>0$ such that the epidemic propagates for \eqref{syst} with $A_I = d_I I_d$ if $d_I<d^\star$ and fades off if $d_I>d^\star$.
\end{cor}
This result gives us situations where increasing the diffusion of the infectious individuals can trigger an epidemic that would fade off if this diffusion were small. This may seem surprising at first. However, one has to keep in mind that increasing the diffusion of the infectious individuals also increases their scattering, which in turns can prevent the formation of clusters of infections.\\

The results above answer Question \ref{q1}. We now turn to Question \ref{q2}, that is, we investigate how the predictions of the diffusive model \eqref{syst} and of the averaged model \eqref{SIRav} are related.

We start with investigating the differences concerning whether or not the epidemic propagates. We have the following:
\begin{req}\label{req av}
Let $\alpha, \mu, S_0$ be positive and continuous on $\ol\O$. Then, owing to Theorem~\ref{th SIR}, we know that the epidemic propagates for the averaged model \eqref{SIRav} if and only if
$$
\frac{\fint_\O \alpha \fint_\O S_0}{\fint_\O \mu}>1.
$$
\end{req}

Owing to Theorem \ref{th threshold} and Proposition \ref{prop quali}, we get the following:

\begin{cor}\label{cor quali}
Let $\alpha,\mu,S_0$ be positive and continuous on $\ol\O$.

\begin{itemize}
    \item If $\alpha, \mu$ are constant, define 
    $$
    \t R_0 := \frac{\alpha\fint_\O S_0}{\mu}.
    $$
    Then the epidemic propagates for both the diffusive model \eqref{syst} and for the averaged model \eqref{SIRav} if $\t R_0>1$ and fades-off if $\t R_0 <1$.
    
    \item If
    $$
    \frac{\fint_\O \alpha \fint_\O S_0}{\fint_\O \mu } >1,
    $$
    then the epidemic propagates for both the diffusive model \eqref{syst} and for the averaged model \eqref{SIRav}.
    
    \item If $\alpha,\mu$ are such that
     $$\frac{\fint_\O\alpha\fint_\O S_0}{\fint_\O \mu}<1<\max_{x\in\O}\left\{\frac{\alpha(x)\fint_\O S_0}{\mu(x)}\right\},$$
     and if $d_I>0$ is small enough, then the epidemic propagates in the diffusive model \eqref{syst} with $A_I = d_I I_d$ but fades off in the averaged model \eqref{SIRav}.
     
\end{itemize}
\end{cor}
Let us comment on this result. The first point can be seen as a direct generalization of the formula of the basic reproduction number of Kermack and McKendrick given by Theorem \ref{th SIR}, it says that, when $\alpha,\mu$ are constant, then the diffusive model \eqref{syst} and the averaged model \eqref{SIRav} agree on whether or not the epidemic propagates.

The second point tells us that this is still partly true when $\alpha$ or $\mu$ is not constant. In this case, if the averaged model \eqref{SIRav} predicts that the epidemic propagates, then so does the diffusive model \eqref{syst}.

However, the third point tells us that the reciprocal fails to hold true in some cases. \\


As a consequence of the third point of the corollary, as soon as either $\alpha$ or $\mu$ are not constant,  we can exhibit situations where the epidemic propagates in the diffusive model \eqref{syst} but fades off for the averaged model \eqref{SIRav}. In particular, owing to Definition \ref{def}, this means that there are situations where, for $S_0$ given, up to taking $I_0$ small enough, the final number of susceptible individuals will be much larger in the averaged model than in the diffusive one.\\

Let us now consider more precisely the case where $\alpha, \mu, S_0$ are constant. Then, as mentioned in Corollary \ref{cor quali}, the diffusive and the averaged model agree on whether or not the epidemic propagates. 

However, our next results says that the two models give a different prediction in what concerns the final number of susceptible individuals.\\




\begin{theorem}\label{th final state}
Let $\alpha,\mu,S_0$ be positive constants. Let $I_0$ be non-negative and continuous on $\ol\O$. Let $(S(t,x),I(t,x))$ be the solution of \eqref{systH} arising from the initial datum $(S_0,I_0)$. Let $(S^A(t),I^A(t))$ be the solution of \eqref{SIRav} arising from the initial datum $(S_0, \fint_\O I_0)$. Then
	$$
	S_{\infty} \geq S_{\infty}^A,
	$$ 
	where $S_\infty := \lim_{t\to+\infty}S(t,x)$ and $S^A_\infty := \lim_{t\to+\infty}S^A(t)$.
	
	Moreover, the inequality is strict as soon as $I_0$ is not constant. If $I_0$ is constant, it is an equality.
\end{theorem}
This result tells us that the averaged model \eqref{SIRav} always underestimate the number of final number of susceptible individuals compared to the homogeneous diffusive model \eqref{systH}. Observe that, if either $\alpha$ or $\mu$ were not constants, then the third point of the above Corollary \ref{cor quali} tells us that we can have the opposite : the averaged model can overestimate the number of infected individuals.

However, this difference disappears in the limit where $I_0$ goes to zero.


\begin{prop}\label{prop equal}
    Let $\alpha,\mu,S_0$ be positive constants. Let $(I^n_0)_{n\in\N}$ be a sequence of non-negative, continuous functions on $\ol\O$ such that $\vert\vert I^n_0\vert\vert_{L^\infty}\to 0$ as $n$ goes to $+\infty$.

    Let $(S_n(t,x),I_n(t,x))$ be the solution of \eqref{systH} with initial datum $(S_0,I^n_0)$ and let $(S^A_n(t),I^A_n(t))$ be the solution of \eqref{SIRav} with initial datum $(S_0,\fint_\O I^n_0)$.

    Let $S_\infty^n := \lim_{t\to+\infty}S_n(t,x)$ and ${S_\infty}^{A,n} := \lim_{t\to+\infty}S_n^A(t)$. Then, there is $C>0$ such that, if $d_S,d_I>C$, we have
    $$
     \lim_{n\to+\infty}S_{\infty}^n  =   \lim_{n \to \infty} {S_{\infty}}^{A,n}.
    $$
\end{prop}
The hypothesis that the diffusivities $d_S,d_I$ should be large in Proposition \ref{prop equal} is purely technical, and we believe that the result holds true without it. We leave it as an open question.\\







The organisation of the paper is the following. In Section \ref{cv sol}, we give some technical results that will be useful in the sequel. In particular, we show that the solutions of \eqref{syst} converge to constant functions on $\O$. In Section \ref{sec prop}, we study the threshold phenomenon. We prove Theorem \ref{th threshold} in Section \ref{sec thr} and Proposition \ref{prop quali} and Corollary \ref{cor prop} in Section \ref{sec quali}. We compare the diffusive and the averaged model in Section \ref{sec limit}. There, we prove Corollary \ref{cor quali}, Theorem \ref{th final state} and Proposition~\ref{prop equal}.

\section{Convergence of solutions}\label{cv sol}

We gather in this section some technical results, that will be useful in the sequel. The main result of this section is that the solutions of \eqref{syst} converge to functions constant on $\O$.



\begin{lemma}\label{lem use}
Let $(S(t,x),I(t,x))$ be the solution of \eqref{syst} arising from the initial datum $(S_0,I_0)$, where $S_0,I_0$ are continuous on $\ol\O$ and non-negative. Then:
\begin{itemize}
    \item If $S_0\not\equiv 0$ and $I_0\not\equiv0$, then $S(t,x)>0, I(t,x) >0$ for all $t>0$, $x\in \O$.
    \item $\int_\O S(t,x)dx + \int_\O I(t,x)dx \leq  \int_\O S_0(x)dx + \int_\O I_0(x)dx$, for all $t>0$.
    
    \item We have, for all $t>0$,
    $$
    \vert\vert S(t,\cdot)\vert\vert_{L^\infty} \leq \vert \vert S_0\vert\vert_{L^\infty}.
    $$
    \item There is $K>0$ independent of $I_0$ such that, for all $t>0$,
   $$
    \vert\vert I(t,\cdot)\vert\vert_{L^\infty} \leq K\vert \vert I_0\vert\vert_{L^\infty}.
    $$
\end{itemize}
\end{lemma}
\begin{proof}
The first point is a direct application of the comparison principle for parabolic equations (see \cite{PW} for instance): indeed, both $S(t,x)$ and $I(t,x)$ are solutions of a parabolic equation with bounded coefficients and with non-negative initial data.

To prove the second point, we define $m(t) := \int_\O S(t,x)dx + \int_{\O}I(t,x)dx$. Then, we have
$$
\dot m(t) = -\int_\O\alpha SI <0,
$$
that is, $t\mapsto m(t)$ is decreasing, hence the result.

The third point comes again from the parabolic comparison principle. Let us give some details on its applications, as it will be used many times in the sequel. Indeed, if $(S,I)$ is the solution of \eqref{syst} arising from the initial datum $(S_0,I_0)$, then the function everywhere constant $v(t,x) := \vert\vert S_0\vert\vert_{L^\infty}$ satisfies the differential inequality
$$
\partial_t v - \nabla (A_S \nabla v) +\alpha v I \geq 0,\quad \text{ for } t>0, \ x\in \O,
$$
with conormal boundary conditions, that is, it is supersolution of a parabolic equations satisfied by $S(t,x)$. Because the initial data are ordered $v(0,\cdot) =  \vert\vert S_0\vert\vert_{L^\infty} \geq S_0$, the parabolic comparison principle (see \cite{PW} for instance) gives us that the functions are ordered for all positive times, that is,
$$
S(t,x) \leq v(t,x) =  \vert\vert S_0\vert\vert_{L^\infty}, \quad\text{ for } \ t>0, \ x\in \O.
$$

Let us prove the fourth point. Let $\beta := \max\{\alpha \vert \vert S_0\vert\vert_{L^\infty}+\mu\}$. The function constant in space $u(t,x) :=  \vert \vert I_0\vert\vert_{L^\infty}e^{\beta t}$ satisfies
$$
\partial_t u -\nabla (A_I \nabla u) -(\alpha S -\mu)u\geq 0, \quad \text{ for } \ t>0, \ x\in \O,
$$
with conormal boundary conditions. Hence, $u$ is supersolution of a parabolic equation satisfied by $I$. Because $u$ and $I$ are ordered at the initial time, the parabolic comparison principle implies that
$$
I(t,x) \leq \vert \vert I_0\vert\vert_{L^\infty}e^{\beta t}, \quad t>0,\ x\in \O.
$$
In particular, this gives the existence of $K>0$ independent of $I_0$ such that $I(t,x)\leq K\vert \vert I_0\vert\vert_{L^\infty}$ for $t\in [0,1)$.

Moreover, the Harnack inequality (see \cite{Evans, Lie} for instance) applied to the parabolic equation solved by $I$,gives us that there is $C>0$ such that
$$
\max_{x\in \O} I(t,x) \leq C  \int_{\O}I(t+1,x)dx \quad  \text{ for }\ t\geq 1.
$$
The constant $C$ in this inequality depends on the $L^\infty$ norm of the coefficients, that is, on the $L^\infty$ norm of $A_I$ and of its first derivative, and on the $L^\infty$ norm of $S(t,x)$. This last quantity is bounded independently of $I_0$, owing to the third point. Because $\int_\O I(t,x)dx$ is bounded by the supremum of $I_0$ (owing to the second point), the fourth point follows.
\end{proof}

The next result shows that the integrals of $S,I$ converge.
\begin{lemma}\label{lemma int}
	There is $S_{\infty} >0$ such that
	$$
	\fint_{\O}S(t,x)dx \underset{t\to +\infty}{\longrightarrow} S_{\infty}, \quad \int_{\O}I(t,x)dx \underset{t\to +\infty}{\longrightarrow} 0.
	$$
\end{lemma}
\begin{proof}
	We denote $\ol S(t) := \fint_{\O}S(t,x)dx$ and $\ol I(t) := \fint_{\O}I(t,x)dx$. Because $\dot{\ol S}(t) = -\int_\O \alpha S I \leq 0$, the function $t\mapsto \ol S(t)$ decays, hence it converges (it is positive for all $t>0$) to some limit that we call $S_\infty$. 
	
	Let $M(t) = \ol S + \ol I$, and let $\ul \mu :=  \min \mu >0$. We find that
	$$
	\dot M(t) + \ul \mu M(t) \leq \ul \mu \ol S(t).
	$$
	Then
	$$
	M(t) \leq M(0)e^{-\ul \mu t} +\ul \mu \int_0^t\ol S(\tau)e^{-\ul \mu (t-\tau)}d\tau.
	$$
	Because $\ol S(t) \to S_\infty$ as $t$ goes to $+\infty$, it is easy to verify that the quantity $\ul \mu \int_0^t\ol S(\tau)e^{-\ul \mu (t-\tau)}d\tau$ also goes to $S_\infty$ as $t$ goes to $+\infty$. Therefore, $\ol I(t) = M(t) - \ol S(t)$ goes to zero as $t$ goes to $+\infty$.
	 \end{proof}

\begin{prop}\label{prop cv unif}
Let $\alpha,\mu, S_0,I_0$ be positive and continuous on $\O$. Let $(S(t,x),I(t,x))$ be the solution of \eqref{syst} arising from the initial datum $(S_0,I_0)$. Then,
$$
S(t,x) \underset{t\to +\infty}{\longrightarrow} S_\infty,\quad I(t,x)\underset{t\to +\infty}{\longrightarrow} 0,
$$
and these convergences hold true uniformly in $\O$.
\end{prop}
\begin{proof}
We already know that $I(t,\cdot)$ goes to zero in $L^1(\O)$ sense as $t$ goes to $+\infty$, owing to Lemma \ref{lemma int}. Therefore, because $S(t,x)$ is uniformly bounded by $\vert\vert S_0\vert\vert_{L^\infty}$, the Harnack inequality for parabolic equations (see \cite{Evans, Lie}) implies that the convergence of $I$ to zero is uniform.

The situation for $S$ is a bit more involved. Let $(t_n)_{n\in \N}$ be a sequence of positive real numbers such that $t_n \to +\infty$ as $n$ goes to $+\infty$. Let $S_n(t,x) := S(t+t_n, x)$. It solves
$$
\partial_t S_n(t,x) = \nabla(A_S(x)\nabla S_n)(t,x) - \alpha(x) S_n(t,x) I(t+t_n,x), \quad t>-t_n,\ x\in \O.
$$
Because $I(\cdot+t_n,\cdot)$ is bounded independently of $n$ and converges to zero, the parabolic regularity estimates (see \cite{Lie} for instance) give us that the sequence $S_n(t,x)$ converges uniformly, as $n$ goes to $+\infty$, to a positive bounded function $\t S_\infty(t,x)$ that solves
$$
\partial_t \t S_\infty = \nabla(A_S(x)\nabla \t S_\infty)(t,x) , \quad t \in \R,\ x\in \O.
$$
However, the only bounded solutions of such a diffusion equation for $t\in \R$ are the constants. Because $\fint_\O S_n(t,x)dx \to S_\infty$ as $n$ goes to $+\infty$, we find that $\t S_\infty = S_\infty$, hence
$$
S(t+t_n,x) \to S_\infty,
$$
as $n$ goes to $+\infty$, uniformly in $x\in \O$. This is true for every diverging sequence $(t_n)_{n\in \N}$. Then, $S(t,\cdot)$ converges uniformly to $S_\infty$ as $t$ goes to $+\infty$.
\end{proof}

\section{Propagation of the epidemic}\label{sec prop}

\subsection{The threshold phenomenon}\label{sec thr}

This section is dedicated to proving Theorem \ref{th threshold}. We start with a lemma that states that the limit state $S_\infty$ is linearly stable.
\begin{lemma}\label{lem thr}
Let $(S(t,x),I(t,x))$ be the solution of \eqref{syst} arising from the initial datum $(S_0,I_0)$, where $S_0,I_0$ are non-negative and continuous on $\ol\O$. Let $S_{\infty} := \lim_{t\to +\infty} S(t,x)$. Then, the principal eigenvalue of the operator acting on $C^0(\ol\O)$,
$$
\phi \to -\nabla\cdot(A_I(x)\nabla \phi) -(\alpha(x) S_\infty -\mu(x))\phi,
$$
is non-negative.
\end{lemma}
\begin{proof}
Let $\Lambda$ be the principal eigenvalue of the operator $-\nabla\cdot A_I\nabla -(\alpha(x) S_\infty -\mu(x))$ and let $\phi$ be a positive principal eigenfunction associated with this eigenvalue, i.e.,
$$
- \nabla\cdot(A_I(x)\nabla \phi(x))  -(\alpha(x)S_\infty -\mu(x))\phi(x) = \Lambda \phi(x),\quad x\in\O,
$$
and $\nu\cdot A_I \nabla \phi =0$ on $\partial \O$. We argue by contradiction: assume that $\Lambda <0$. Let $\e>0$ be small enough so that $\Lambda + \e\vert\vert \alpha \vert\vert_{L^{\infty}}\leq 0$. Because the convergence of $S(t,x)$ to $S_\infty$ is uniform, owing to Proposition \ref{prop cv unif}, there is $T>0$ such that
$$
S(t,x) \geq S_\infty -\e, \quad \text{for}\ t\geq T.
$$
Therefore, for $t\geq T$, we have 
$$
- \nabla(A_I(x)\nabla \phi)  -(\alpha(x)S(t,x) -\mu(x))\phi \leq  (\Lambda + \e\vert\vert \alpha \vert\vert_{L^{\infty}}) \phi \leq 0.
$$
This means that, for $t\geq T$, the function $\phi$ is a stationary subsolution of the parabolic equation satisfied by $I(t,x)$. The parabolic comparison principle implies that, for $\eta>0$ small enough so that $\eta \max \phi \leq  \min I(T,\cdot)$, we have
$$
0<\eta \phi(x) \leq  I(t,x), \quad \text{ for }\ t\geq T, \ x\in \O.
$$
This is in contradiction with the fact that $I(t,\cdot)$ goes to $0$ as $t$ goes to $+\infty$, hence the result.
\end{proof}
We are now in position to prove Theorem \ref{th threshold}.

\begin{proof}[Proof of Theorem \ref{th threshold}.]
Let $\alpha,\mu,S_0,A_I$ be as in the statement of the theorem. Let $I_0$ be continuous and non-negative on $\ol\O$.

In the whole proof, we let $(S,I)$ be the solution of \eqref{syst} arising from the initial datum $(S_0,I_0)$ and we define $S_\infty := \lim_{t\to+\infty}S(t,x)$.

We also let $\l$ denote the principal eigenvalue of the operator $\phi \mapsto -\nabla \cdot( A_I\nabla \phi) -\left(\alpha \left(\fint S_0\right) - \mu\right)\phi$ and we let $\Lambda$ denote the principal eigenvalue of the operator $\phi \mapsto -\nabla \cdot( A_I\nabla \phi ) -(\alpha S_\infty-\mu)\phi$.

\medskip
\emph{Step $1$. $\l <0 \implies$ Propagation.}

Let $\psi, \phi$ be positive principal eigenvalues associated to $\l$ and $\Lambda$ respectively. Then
$$
-\nabla \cdot( A_I\nabla  \psi ) -\left(\alpha \fint S_0-\mu\right)\psi = \l \psi,\quad -\nabla \cdot( A_I\nabla  \phi) -(\alpha S_\infty-\mu)\phi = \Lambda\phi. 
$$
We have
$$
 -\nabla \cdot( A_I\nabla \psi) -(\alpha S_{\infty}-\mu)\psi -\alpha\left(\fint S_0-S_\infty\right)\psi = \l \psi.
$$
We multiply by $\phi$ and we integrate on $\O$ to obtain
$$
 \Lambda \int \phi \psi - \left(\fint S_0-S_\infty\right)\int \alpha \phi\psi = \l \int \phi \psi ,
$$
Owing to the positivity of the principal eigenfunctions, we get
$$
\fint S_0 - S_\infty \geq \frac{\Lambda - \l}{\max \alpha}.
$$
Therefore, because $\Lambda\geq 0$,
$$
S_{\infty} \leq \fint S_0 - \frac{\vert \l \vert}{\max \alpha}.
$$
The quantity $\frac{\vert \l \vert}{\max \alpha}$ is strictly positive and does not depend on $I_0$. We have thus proven that the epidemic propagates.

\medskip
\emph{Step $2$. $\l > 0 \implies$ Extinction.}\\
Let $\e  >0$ be fixed. We define $K := \max\{\alpha S_0 -\mu\}$. The parabolic comparison principle gives us that
\begin{equation}\label{est I}
I(t,x) \leq \vert\vert I_0\vert\vert_{L^{\infty}} e^{Kt}, \quad \text{ for }\ t>0, \ x\in \O.
\end{equation}
Now, let $h(t,x)$ be the solution of
$$
\partial_t h - \nabla\cdot(A_S\nabla h) =0,\quad t>0, \ x\in \O,
$$
with conormal Neuman boundary conditions and with initial datum $h(0,\cdot) = S_0$. Because $h$ is supersolution of the equation satisfied by $S$ (because $I$ is non-negative), the parabolic comparison principle implies that $S(t,x) \leq h(t,x)$ for $t>0, x\in \O$. Because $h\to \fint_\O S_0$ uniformly as $t$ goes to $+\infty$, we can find $T>0$ independent of $I_0$ such that
$$
S(t,x)  \leq \fint_\O S_0 + \frac{\l}{2\max \alpha}, \quad t>T, \ x\in \O.
$$
This estimate at hand, observe that the function $w(t,x) := \phi(x)e^{-\frac{\l}{2} t}$ satisfies, for $t>T$ and $x\in \O$,
$$
\partial_t w -\nabla \cdot (A_I \nabla w) -(\alpha S -\mu)w \geq  \left(-\frac{\l}{2} + \l - \alpha \frac{\l}{2 \max \alpha}\right)e^{-\frac{\l}{2} t} \phi\geq 0.
$$
Then, $w$ is supersolution of the parabolic equation satisfied by $I$ for $t>T$. The parabolic comparison principle gives us that
$$
I(t,x)\leq \left(\frac{\max_{x\in\O} I(T,x)}{\min_{x\in\O}\phi} \right)e^{-\frac{\l}{2} (t-T)} \phi(x),\quad \text{ for } t>T,\ x\in \O.
$$
Combining this with \eqref{est I}, we find that
$$
I(t,x) \leq C \vert \vert I_0 \vert\vert_{L^{\infty}} e^{-\frac{\l}{2} t},\quad \text{ for }\ t>0,\ x\in \O,
$$
for some $C>0$ independent of $I_0$.

Now, let $u(t,x)$ be the solution of
$$
\partial_t u = \nabla(A_S\nabla u)  - C \vert \vert I_0 \vert\vert_{L^{\infty}} e^{-\frac{\l}{2} t}u, \quad t>0,\ x\in \O,
$$
with initial datum $S_0$ and with conormal Neuman boundary conditions. Because $u$ is subsolution of the equation satisfied by $S$, the parabolic comparison principle gives us
\begin{equation}\label{eq S u}
u(t,x) \leq S(t,x),\quad \text{ for }\ t>0, \ x\in \O.
\end{equation}
Observe that
$$
w(t,x) := u(t,x) e^{\frac{2C \vert \vert I_0 \vert\vert_{L^{\infty}}}{\l}(1-e^{-\frac{\l}{2} t})}
$$
solves $\partial_t w = \nabla (A_S \nabla w) $ on $\O$ with conormal Neuman boudary conditions and with initial condition $S_0(x)$. Therefore
$$
u(t,x) \underset{t \to +\infty}{\longrightarrow} \left( \fint_\O S_0 \right)e^{-\frac{2C \vert \vert I_0 \vert\vert_{L^{\infty}}}{\l}}.
$$
Hence, taking the limit $t\to+\infty$ in \eqref{eq S u}, we have
$$
S_{\infty} \geq  \left( \fint_\O S_0 \right)e^{-\frac{2C \vert \vert I_0 \vert\vert_{L^{\infty}}}{\l}}.
$$ 
Up to taking $ \vert \vert I_0 \vert\vert_{L^{\infty}}$ small enough, we can ensure that 
$$
S_{\infty} \geq   \fint_\O S_0 -\e,
$$
hence the result.
\end{proof}

\subsection{Qualitative properties}\label{sec quali}

We now turn to the proof of Proposition \ref{prop quali} and Corollary \ref{cor prop}.

\begin{proof}[Proof of Proposition \ref{prop quali}]
We let $\lambda_1$ denote the principal eigenvalue of the elliptic operator
$$L : \ \phi \mapsto -\nabla(A_I \nabla \phi) -\left(\alpha\fint_\O S_0 -\mu\right)\phi.$$
The classical Rayleigh formula (see \cite{Br}) gives us
\begin{equation}\label{Rayleigh}
    \lambda_1 = \min_{\psi \in H^1(\O)} \frac{\int_\O (A_I(x)\nabla \psi \cdot \nabla \psi) - \left(\fint_\O S_0\right) \int_\O\alpha \psi^2 + \int_\O \mu \psi^2}{\int_\O \psi^2}.
\end{equation}
\medskip
\emph{Proof of the first point.}\\
We only prove the monotony with respect to the $A_I$ argument, the proof is similar for the other parameters.

Let $A, B$ be two elliptic matrices such that $A\leq B$ and denote $V := \alpha \fint_\O S_0 - \mu$. Let $\l(A), \l(B)$ denote the principal eigenvalues of $-\nabla A \nabla - V$ and of $-\nabla B \nabla - V$ respectively. Let $\phi_B$ denote a principal eigenfunction associated with $\l(B)$. Using $\phi_B$ as a test function in the Rayleigh formula giving $\l(A)$, we get
\begin{equation*}
\begin{array}{rl}
\l(A) &\leq   \frac{\int_\O (A \nabla \phi_B \cdot \nabla \phi_B) -\int_\O V\phi_B^2}{\int_\O \phi_B^2} \\
&= \frac{\int_\O ((A -B) \nabla \phi_B \cdot \nabla \phi_B) }{\int_\O \phi_B^2}   +  \frac{\int_\O (B \nabla \phi_B \cdot \nabla \phi_B)   -\int_\O V\phi_B^2}{\int_\O \phi_B^2}\\
&\leq \l(B).
\end{array}
\end{equation*}
In addition, if $A<B$, then this inequality is strict if and only if $\nabla \phi_B \not\equiv 0$, which is the case if and only if $V$ is constant, that is, if $\alpha \fint_\O S_0 -\mu$ is constant.

\medskip
\emph{Proof of the second point.}\\
We denote $\lambda_1(d_I), \phi_{d_I}$ the principal eigenvalue and eigenfunction of the operator $L$, that is,
\begin{equation}\label{eq ra}
-d_I \Delta \phi_{d_I} - (\alpha\fint_\O S_0 -\mu)\phi_{d_I} = \lambda_1(d_I)\phi_{d_I}.
\end{equation}
We normalize $\phi_{d_I}$ so that $\int_\O \phi_{d_I}^2 = 1$.

Observe first that, using the constant function $\psi =1$ in \eqref{Rayleigh}, we have that $\lambda_1(d_I) \leq -\fint\alpha \fint S_0 +\fint \mu$. 

Multiplying \eqref{eq ra} by $\phi_{d_I}$ and integrating on $\O$, we find that
$$
\vert\vert \nabla \phi_{d_I}\vert\vert_{L^2(\O)}^2\leq \frac{\max\{\alpha \fint S_0 -\mu \} -\fint\alpha \fint S_0 +\fint \mu}{d_I}. 
$$
Therefore, up to extraction, we have that $\phi_{d_I}$ goes to a constant as $d_I$ goes to $+\infty$ (up to a subsequence) in the $L^2$ norm. Integrating \eqref{eq ra} on $\O$, we get
$$
-\int_{\O} (\alpha\fint S_0 -\mu)\phi_{d_I} = \lambda_1(d_I)\int_\O \phi_{d_I}.
$$
Taking the limit $d_I\to +\infty$, we get the result.

\medskip
\emph{Proof of the third point.}\\
First, \eqref{Rayleigh} implies that
$$
\lambda_1(d_I)\geq -\max\{\alpha \fint S_0 -\mu \}.
$$
Now, let $x_m$ be such that $\alpha(x_m) \fint S_0 -\mu(x_m) = \max\{\alpha \fint S_0 -\mu \}$ and denote $\delta_{x_m}$ the Dirac mass centered at $x_m.$ Let $(\psi_n)_{n\in\N}$ be a sequence of functions such that
$$
\psi^2_n \to \delta_{x_m}, \quad \int_\O \psi_n^2 =1,
$$
where the convergence holds in the distribution sense. Then, using $\psi_n$ as test function in \eqref{Rayleigh}, we get
$$
\limsup_{d_I \to 0}\lambda_1(d_I) \leq -\int_\O (\alpha \fint S_0 -\mu)\psi_n^2.
$$
Taking the limit $n \to +\infty$, we get
$$
\limsup_{d_I \to 0}\lambda_1(d_I) \leq -\max\{\alpha \fint S_0 -\mu \},
$$
hence the result.
\end{proof}

We can now get Corollary \ref{cor prop}.
\begin{proof}[Proof of Corollary \ref{cor prop}]
Let $\alpha,\mu,S_0$ be positive and continuous on $\ol \O$ be such that    $$\frac{\fint_\O\alpha\fint_\O S_0}{\fint_\O \mu}<1<\max_{x\in\O}\left\{\frac{\alpha(x)\fint_\O S_0}{\mu(x)}\right\}.$$
This implies that
\begin{equation}\label{eq nc}
\min_{x\in\O}\left\{\mu(x)-\alpha(x)\fint S_0\right\}<0<\fint \mu - \fint \alpha \fint S_0.
\end{equation}
Let $\lambda_1(d_I)$ denote the principal eigenvalue of the operator
$$
L : \phi \mapsto -d_I\Delta \phi -(\alpha \fint_\O S_0 -\mu)\phi.
$$
Then, owing to Proposition \ref{prop quali}, we know that $d_I \mapsto \l(d_I)$ is strictly increasing (because $\alpha \fint_\O S_0 -\mu$ can not be constant owing to \eqref{eq nc}) and is such that
$$
\lim_{d_I \to 0} \l(d_I)<0<\lim_{d_I\to +\infty}\l(d_I).
$$
Corollary \ref{cor prop} follows by defining
$$
d^\star := \sup\{ d >0 \ : \ \l(d)\leq 0\},
$$
thanks to Theorem \ref{th threshold}.
\end{proof}

\section{Comparing the diffusive and the averaged model}\label{sec limit}

We now focus on Question \ref{q2}, that is, we compare how the diffusive model \eqref{syst} and the averaged model \eqref{SIRav} differ.

Our first result on this question, Corollary \ref{cor quali}, directly comes from Theorem \ref{th threshold} and Proposition \ref{prop quali}.
\begin{proof}[Proof of Corollary \ref{cor quali}]
Let $\alpha, \mu, S_0$ be positive and continuous on $\ol\O$.

The first point comes directly by observing that, if $\alpha, \mu$ are constant, then the principal eigenvalue of the elliptic operator
$$
-\nabla\cdot A_I \nabla -\left(\alpha\fint_\O S_0 -\mu\right)
$$
is $-\alpha\fint_\O S_0 +\mu$. Combining this with Theorem \ref{th threshold}, we get the first point.

The second and third points come directly from the second and third points of Proposition \ref{prop quali} combined with Theorem \ref{th threshold}, together with Remark \ref{req av}.
\end{proof}

We now prove Theorem \ref{th final state}. 

\begin{proof}[Proof of Theorem \ref{th final state}]
Let $S_0, \mu, \alpha$ be positive constants. Let $I_0$ be non-negative and continuous on $\ol\O$. Let $(S(t,x),I(t,x))$ be the solution of \eqref{systH} arising from the initial datum $(S_0,I_0)$.

Let us start with proving that, for every $0 < T_1 \leq T_2$, we have
\begin{multline}\label{eq T12}
\int_\O f(S(T_2,x))dx + \frac{\alpha}{\mu}\int_\O I(T_2,x)dx =\\ \int_\O f(S(T_1,x))dx + \frac{\alpha}{\mu}\int_\O I(T_1,x)dx + d_S\int_\O\int_{T_1}^{T_2}\frac{\vert \nabla S\vert^2(\tau,x)}{S(\tau,x)}d\tau dx,
\end{multline}
where $f(x) := \frac{\alpha}{\mu}x - \ln(x)$.

We start with observing that
$$
\partial_t \ln(S(t,x)) = d_S \frac{\Delta S }{S} - \alpha I. 
$$
We integrate for $x\in \O$ and for $t=T_1$ to $t=T_2$ to get
$$
\int_\O \ln(S(T_2,x))dx - \int_\O\ln(S(T_1,x))dx = d_S\int_ {T_1}^{T_2}\int_{\O} \frac{\vert \nabla S\vert^2}{S^2}(\tau,x)d\tau dx - \alpha \int_{T_1}^{T_2}\int_{\O}I(\tau,x)d\tau dx.
$$
Combining the equations for $I$ and $S$ we have
$$
\partial_t I - d_I\Delta I +\mu I = -(\partial_t S -d_S \Delta S).
$$
Integrating this for $x\in \O$ and for $t=T_1$ to $t=T_2$ we obtain
$$
\int_\O (S(T_2,x)+I(T_2,x))dx + \mu \int_{T_1}^{T_2}\int_\O I(\tau,x)d\tau dx =  \int_\O S(T_1,x)dx+\int_\O I(T_1,x)dx. 
$$
Combining what precedes, we obtain \eqref{eq T12}.

Now, taking the limits $T_2 \to +\infty$ and $T_1 \to 0$ in \eqref{eq T12}, we obtain
$$
 f(S_\infty) =  f(S_0) + \frac{\alpha}{\mu}\fint_\O I_0 + d_S\fint_\O\int_{0}^{+\infty}\frac{\vert \nabla S\vert^2(\tau,x)}{S(\tau,x)}d\tau dx.
$$
Owing to the non-negativity of $d_S\fint_{\O}\int_0^{+\infty} \frac{\vert \nabla S\vert^2}{S^2}$, we have
$$
f(S_\infty) \leq f\left( S_0\right) + \frac{\alpha}{\mu}\fint_\O I_0,
$$
and this inequality is strict as soon as $\nabla S \not\equiv 0$, which is the case if and only if $I_0$ is not constant.

Now, let $(S^A,I^A)$ be the solution of the SIR system \eqref{SIR} with initial datum $(S_0,\fint_\O I_0)$, and let $S_\infty^A = \lim_{t\to +\infty} S^A(t)$. As recalled in Theorem \ref{th SIR}, we have
$$
f(S_\infty^A) = f(S_0) + \frac{\alpha}{\mu}\fint_\O I_0.
$$
Therefore, because $f$ is strictly decreasing on $(0,\frac{\mu}{\alpha}]$ and because $S_\infty, S_\infty^A \leq \frac{\mu}{\alpha}$, we obtain
$$
S_\infty \leq S_\infty^H,
$$
and this inequality is strict as soon as $I_0$ is not constant.
\end{proof}

We now turn to Proposition \ref{prop equal}. In the course of the proof, we will need the two following technical lemmas:
\begin{lemma}\label{lem 1}
Let $S_0, \alpha, \mu$ be positive constant.
For any $\delta >0$, there is $\eta >0$ such that, for every $I_0$ such that $\vert\vert I_0\vert\vert_{L^\infty} \leq \delta$, the solution $(S,I)$ of \eqref{systH} arising from the initial datum $(S_0,I_0)$, satisfies
$$
S(t,x)\geq \eta, \quad \text{for all } \ t>0, \ x\in \O.
$$
\end{lemma}
\begin{proof}
We argue by contradiction. Assume that there is $\delta >0$ and a sequence $(I_0^n)_{n\in \N}$ with $\vert\vert I_0^n\vert\vert_{L^\infty}\leq \delta$, such that, denoting $(S^n,I^n)$ the solution of \eqref{systH} arising from the initial datum $(S_0,I_0^n)$, there are $t_n>0$, $x_n \in \O$ such that
$$
S^n(t_n,x_n)\leq \frac{1}{n}.
$$
Let us start with observing that, owing to Lemma \ref{lem use}, we have, for some $K>0$, $I^n(t,x) \leq K \delta$, hence
$$
\partial_t S^n - d_S \Delta S^n +\alpha K \delta S^n \geq 0.
$$
The parabolic comparison principle implies that
$$
S^n(t,x)\geq S_0 e^{-\alpha K \delta t}, \quad \text{for } \ t>0, \ x\in \O.
$$
This implies that, necessarily, $t_n \to +\infty$ as $n$ goes to $+\infty$.


Now, it follows from the relation \eqref{eq T12} that, for every $n\in \N$ and $t>0$,
$$
\int_\O f(S^n(t,x))dx + \frac{\alpha}{\mu}\int_\O I^n(t,x)dx \leq f(S_0) \vert \O\vert + \frac{\alpha}{\mu}\int_\O I^n_0,
$$
hence
\begin{equation}\label{eq min s}
\int_\O f(S^n(t,x))dx  \leq f(S_0) \vert \O\vert + \frac{\alpha}{\mu}\delta \vert \O\vert.
\end{equation}
Now, because $S^n$ solves a parabolic equation, the Harnack inequality implies that there is $C>0$ such that, for $n$ large enough so that $t_n>1$,
$$
\max_{x\in \O}S^n(t_n-1,x) \leq C \min_{x\in \O}S^n(t_n,x) \leq \frac{C}{n}.
$$
Therefore, $S^n(t_n-1,x)$ goes to zero uniformly in $x$ as $n$ goes to $+\infty$. However, applying \eqref{eq min s} at $t = t_n -1$ and taking $n$ large enough leads to a contradiction, because $f(x) \to +\infty$ when $x$ goes to $0$. This concludes the proof.
\end{proof}

\begin{lemma}\label{lem 2}
Let $\alpha, \mu, S_0$ be positive constants. Let $I_0\in C^0(\ol \O)$ be non-negative. Let $(S,I)$ be the solution of \eqref{systH} arising from the initial datum $(S_0,I_0)$. Then, there is $a \in \R$ such that, for every $T > 0$,
$$
\int_{\O}\vert \nabla S(t,x) \vert^2dx  \leq   \left(\int_{\O}\vert \nabla I(T,x) \vert^2 + \vert \nabla S(T,x) \vert^2 dx\right) e^{-a( t-T)},\quad \text{ for }\ t\geq T.
$$
Moreover, there is $C>0$ such that, if $d_S,d_I > C$ and if $\vert\vert I_0\vert\vert_{L^\infty} \leq \frac{1}{C}$, we have $a>0$.
\end{lemma}
\begin{proof}
	We define
	$$
	m_S(t) := \frac{1}{2}\int_{\O}\vert \nabla S(t,x) \vert^2dx, \quad m_I(t) := \frac{1}{2}\int_{\O}\vert \nabla I(t,x) \vert^2dx.
	$$
	We have
	\begin{equation*}
	\begin{array}{rl}
		\dot m_S (t) &= -\int_\O\Delta S \partial_t S \\
	&= -d_S\int_\O(\Delta S)^2+\alpha\int_\O  \Delta S S I\\
	&= -d_S \int_\O(\Delta S)^2 - \alpha\int_\O \vert\nabla S \vert^2I - \int_\O S \nabla S \nabla I.
	\end{array}
	\end{equation*}
	Letting $\rho_1 >0$ be the first non-zero eigenfunction of the Laplace operator on $\O$ with Neuman boundary conditions, we have $\rho_1\int_\O \vert \nabla S(t,x)\vert^2dx \leq \int_\O (\Delta S(t,x))^2dx$. Using Lemma \ref{lem use}, we get
$$
\dot m_S(t) + d_S \rho_1 m_S(t)  \leq \alpha \vert\vert S_0\vert\vert_{L^\infty}(m_S(t) + m_I(t)), \quad \text{ for } t>0.
$$
Now, a similar computation for $m_I$ shows that
$$
\dot m_I + (d_I \rho_1 +\mu )m_I \leq \alpha K\vert\vert I_0\vert\vert_{L^\infty} (m_S + m_I) + \alpha \vert \vert S_0\vert\vert_{L^\infty}m_I,
$$
where $K>0$ is from Lemma \ref{lem use}.

Therefore, letting $u(t) := m_S(t)+m_I(t)$, we get
$$
\dot u(t)+ a u(t) \leq 0, \quad \text{ for } t > 0,
$$
where $a:= \min\{ d_S\rho_1 - \alpha \vert\vert S_0\vert\vert_{L^\infty} - \alpha K \vert\vert I_0\vert\vert_{L^\infty}  ,d_I\rho_1 + \mu  - 2\alpha \vert\vert S_0\vert\vert_{L^\infty} - \alpha K \vert\vert I_0\vert\vert_{L^\infty} \}$. Therefore, we conclude that, for $T>0$,
$$
u(t) \leq u(T)e^{-a (t-T)}, \quad \text{ for }\ t\geq T.
$$
We can obseve that, if $d_S > \frac{ \alpha \vert\vert S_0\vert\vert_{L^\infty}}{\rho_ 1}$ and $d_I> \frac{ 2\alpha \vert\vert S_0\vert\vert_{L^\infty}}{\rho_1}$, up to taking  $\vert\vert I_0\vert\vert_{L^\infty} \leq \frac{\mu}{\alpha K}$, we have that $a>0$.
\end{proof}

We are now in position to prove Proposition \ref{prop equal}.
\begin{proof}[Proof of Proposition \ref{prop equal}]
Let $S_0,\alpha,\mu$ be positive constants. Let $(I_0^n)_{n\in \N}$ be a sequence of non-negative continuous functions such that $\vert\vert I_0^n\vert\vert_{L^{\infty}} \to 0$ as $n$ goes to $+\infty$. Let $(S^n,I^n)$ be the solution of \eqref{systH} arising from the initial datum $(S_0,I_0^n)$.

According to Lemma \ref{lem use}, we have $S^n(t,x)\leq  S_0$ and $I^n(t,x)\leq K \vert \vert I^n_0\vert\vert_{L^\infty}$, for some $K>0$ independent of $n$. Therefore, the parabolic comparison principle gives us
$$
I^n(t,x)\leq \vert \vert I^n_0\vert\vert_{L^\infty}e^{(\alpha  S_0 + \mu)t} , \quad \text{ for } t>0,
$$
and
$$
S^n(t,x) \geq S_0e^{-\alpha K \vert \vert I^n_0\vert\vert_{L^\infty} t},\quad \text{ for } t>0.
$$
Therefore, $I^n$ goes to zero and $S^n$ goes to $S_0$ as $n$ goes to $+\infty$, uniformly in $x$ and locally uniformly in $t>0$.

Now, because $I^n$ and $S^n$ are solutions of parabolic equations with bounded coefficients, the parabolic regularity estimates (see \cite{Lie}) gives us that $\nabla S^n$ and $\nabla I^n$ converge to zero, uniformly in $x\in \O$ and locally uniformly in $t>0$.

Let
$$
S_\infty^n := \lim_{t\to +\infty}S^n(t,x).
$$
Using the relation \eqref{eq T12}, with $T_2 \to +\infty$ and with $T_1 = 1$, we find that
\begin{equation}\label{eq lim n}
f(S_\infty^n) = d_S\int_{1}^{+\infty}\fint_\O \frac{\vert \nabla S^n\vert^2}{{S^n}^2} + \fint_\O f(S^n(1,x))dx + \frac{\alpha}{\mu}\fint_\O I^n(1,x)dx.
\end{equation}
Owing to Lemma \ref{lem 1} and Lemma \ref{lem 2}, we get that, if $d_S,d_I>C$, for some $C>0$ given by Lemma \ref{lem 2},
\begin{multline*}
\int_{1}^{+\infty}\fint_\O \frac{\vert \nabla S^n\vert^2}{{S^n}^2} \leq \frac{1}{\min_{x\in \O}S^n}\int_{1}^{+\infty}\fint_\O \vert \nabla S^n\vert^2\\
\leq  \frac{1}{\min_{x\in \O}S^n}\frac{1}{a}\fint_\O (\vert \nabla S^n(1,x)\vert^2+\vert \nabla I^n(1,x)\vert^2)dx,
\end{multline*}
where $a>0$ is from Lemma \ref{lem 2}. Because $\nabla S^n(1,x)$ and $\nabla I^n(1,x)$ go to zero uniformly in $x\in\O$ as $n$ goes to $+\infty$, we find that
$$
\int_{1}^{+\infty}\fint_\O \frac{\vert \nabla S^n\vert^2}{{S^n}^2} \underset{n\to+\infty}{\longrightarrow}0.
$$
Moreover, we have
$$
\fint_\O f(S^n(1,x))dx + \frac{\alpha}{\mu}\fint I^n(1,x)dx \underset{n\to +\infty}{\longrightarrow} \fint_\O f(S_0)dx.
$$
Hence, taking the limit $n\to+\infty$ in \eqref{eq lim n} gives
$$
f(S_\infty^n) \underset{n\to+\infty}{\longrightarrow} f(S_0).
$$
Now, let $(S_n^A(t),I_n^A(t))$ be the solution of the SIR model \eqref{SIR} arising from the initial datum $(S_0,\fint_\O I^n_0)$, and let $S_\infty^{A,n} := \lim_{t\to+\infty}S_n^A(t)$. Owing to Theorem \ref{th SIR}, we know that
$$
f(S_\infty^{A,n}) = f(S_0) + \frac{\alpha}{\mu}\fint I^n_0,
$$
from which we eventually concludes that
$$
f(\lim_{n\to +\infty} S_\infty^{A,n}) = f(\lim_{n\to +\infty}S_\infty^n).
$$
We have that $\lim_{n\to +\infty} S_\infty^{A,n} \leq \frac{\mu}{\alpha}$ and $\lim_{n\to +\infty}S_\infty^n\leq \frac{\mu}{\alpha}$. Therefore, because $f$ is injective on $(0,\frac{\mu}{\alpha})$, the result follows.
\end{proof}


\end{document}